\documentclass{amsart}
\usepackage{amssymb}
\theoremstyle{plain}
\newtheorem{theorem}{Theorem}
\newtheorem*{citedtheoremA}{Theorem A}
\newtheorem*{citedtheoremB}{Theorem B}
\newtheorem*{citedtheoremC}{Theorem C}
\newtheorem*{citedtheoremD}{Theorem D}
\newtheorem*{lemma}{Lemma}

\begin{document}

\title{Blaschke Products with Derivative in Function Spaces}

\author{David Protas}
\address{Department of Mathematics \\
         California State University \\
         Northridge, California 91330}
\email{david.protas@csun.edu}

\date{}

\begin{abstract}
Let $B$ be a Blaschke product with zeros $\{a_n\}$. If $B' \in A^p_{\alpha}$ for certain $p$ and $\alpha$, it is shown that $\sum_n (1 - |a_n|)^{\beta} < \infty$ for appropriate values of $\beta$. Also, if $\{a_n\}$ is uniformly discrete and if $B' \in H^p$ or $B' \in A^{1+p}$ for any $p \in (0,1)$, it is shown that $\sum_n (1 - |a_n|)^{1-p} < \infty$.
\end{abstract}

\maketitle

\section{Preliminaries}
If $f$ is analytic in the open unit disc $U$ and $0 < p < \infty$, then
\[M_p(r; f) = \left\{\frac{1}{2\pi} \int_0^{2\pi} |f(re^{it})|^p\,dt\right\}^{1/p}\]
is defined for each positive $r < 1$. The Hardy space $H^p$ is the set of all functions $f$ analytic in $U$ for which $\|f\|_{H^p} = \sup_{0<r<1} M_p(r; f)$ is finite. Let $dA(z)$ denote Lebesgue area measure. If $f$ is analytic in $U$, $0 < p < \infty$ and $\alpha > -1$, then $f$ is said to be in the space $A_{\alpha}^p$ if 
\[\|f\|_{A_{\alpha}^p} = \left\{\frac{1}{\pi} \iint_U |f(re^{it})|^p (1-r)^{\alpha}\,dA(z)\right\}^{1/p}\] 
is finite. Put $A^p = A_0^p$.

If $\{a_n\}$ is a sequence of complex numbers such that $0<|a_n|<1$ for all $n = 1,2, \ldots$ and $\sum_{n = 1}^{\infty} (1-|a_n|) < \infty$, the Blaschke product
\[B(z) = \prod_{n = 1}^{\infty} \frac{\bar{a}_n}{|a_n|} \frac{a_n - z}{1 - \bar{a}_nz}\]
is an analytic function in $U$ with zeros $\{a_n\}$. A sequence $\{a_n\}$ of points in $U$ is said to be separated or uniformly discrete if there is a constant $\delta > 0$ such that $ \rho(a_m,a_n) \ge \delta$ for all $m \ne n$, where $\rho$ is the pseudohyperbolic metric in $U$ and is given by
\[ \rho(z,w) = \left|\frac{z-w}{1-\bar{w}z}\right|,\quad z, w \in U. \]
The sequence $\{a_n\}$ is said to be uniformly separated if there is a constant $\delta > 0$ such that
\begin{equation} \label{unifsep}
\inf_n \prod_{m \ne n} \rho(a_m,a_n) \ge \delta.
\end{equation}
A Blaschke product whose zeros are uniformly separated is called an interpolating Blaschke product. It is clear that uniformly separated sequences form a proper subset of the set of all uniformly discrete sequences.

We will be considering the relationship between the condition $\sum_n (1-|a_n|)^{\beta} < \infty$ and the inclusion of $B'$ in various spaces $A_{\alpha}^p$ and $H^p$. Since we are interested only in infinite Blaschke products, we can restrict our attention to $0 < \beta \le 1$.

For any $h(r) \ge 0$, we write $h(r) = o(1/(1-r)^l)$ if $\lim_{r\to 1^-} h(r)(1-r)^l = 0$, and $h(r) = O(1/(1-r)^l)$ if $h(r)(1-r)^l$ is bounded for $r \in (0,1)$. Also, we write $K_1 \gtrsim K_2$ if there exists a constant $C>0$ such that $K_1 \ge CK_2$, and we write $K_1 \asymp K_2$ if $K_1 \gtrsim K_2$ and $K_2 \gtrsim K_1$.

\section{Blaschke Sequences}

In \cite{Ki}, H.~Kim proved that if $\alpha > -1$, $\max((\alpha + 2)/2, \alpha + 1) < p < \alpha +2$, and  $B$ is a Blaschke product with zeros $\{a_n\}$ such that 
\[ \sum_n (1-|a_n|)^{2 - p + \alpha} < \infty, \]
then $B' \in A^p_{\alpha}$. In the opposite direction, P.~Ahern and D.~Clark proved in \cite{AC2} that if $-1 < \alpha < -1/2$, $p = 1$, and $B' \in A^p_{\alpha}$, then $\sum_n (1-|a_n|)^{\beta} < \infty$ for all $\beta > (1 + \alpha)/(-\alpha)$. We generalize this to other values of  $p$.

\begin{theorem} \label{sufficient} Suppose $-1 < \alpha < -1/2$, $3/2 + \alpha < p \le 1$, and  $B' \in A^p_{\alpha}$. Then
\begin{equation} \label{betaSum}
\sum_n (1-|a_n|)^{\beta} < \infty
\end{equation}
for all $\beta > \frac{2 - p + \alpha}{p - \alpha -1}$.
\end{theorem}
\begin{proof}
Note that for $p \le 1$ and $-1 < \alpha < -1/2$, we have $p > 3/2 + \alpha$ if and only if $\frac{2 - p + \alpha}{p - \alpha -1} < 1$. Thus, condition~\eqref{betaSum} is stronger than $\{a_n\}$ just being a Blaschke sequence. Following Ahern and Clark, we let $\beta_0 = \inf\left\{\beta : \sum_n (1-|a_n|)^{\beta} < \infty \right\}$.  Pick any $\beta$ such that  $\beta_0 \le \beta \le 1$ and $\sum_n (1-|a_n|)^{\beta} < \infty$, and let $\rho_n = (1 - |a_n|)^{\beta - 1}$. Then $\rho_n \ge 1$ for all $n = 1, 2, \ldots$ . Also, $\epsilon_0 \equiv \prod_n |a_n|^{\rho_n} > 0$ since $\sum \rho_n (1-|a_n|) = \sum_n (1-|a_n|)^{\beta} < \infty$. For each $n = 1, 2, \ldots$, set $B_n(z) = \prod_{j=1}^{n-1} (z-a_j)/(1-\bar{a}_jz)$. Note that 
\[ \frac{r - |a_j|}{1 - |a_j|r} \ge |a_j|^{\rho_j} \Longleftrightarrow r \ge \frac{|a_j| + |a_j|^{\rho_j}}{1 + |a_j|^{\rho_j+1}}, \]
$0 < r < 1$. We put $r_j = (|a_j| + |a_j|^{\rho_j})/(1 + |a_j|^{\rho_j+1})$. Then for each $n = 1, 2, \ldots $ ,
\[ |B_n(z)| \ge \prod_{j=1}^{n-1} \frac{|z|-|a_j|}{1-|a_j||z|} \ge \prod_{j=1}^{n-1} |a_j|^{\rho_j} \ge \prod_{j=1}^{\infty} |a_j|^{\rho_j} = \epsilon_0 \]
for all $z$ such that $|z| \ge r_{n-1}$ since $\{r_j\}$ will be an increasing sequence if we assume, as we can, that $\{|a_j|\}$ is increasing. As pointed out by Ahern and Clark in \cite{AC2},
\[ \frac{1 - |B(z)|^2}{1 - |z|^2} = \sum_n|B_n(z)|^2 \frac{1 - |a_n|^2}{|1 - \bar{a}_nz|^2} \]
for any Blaschke product $B$. Also, Ahern proved in \cite{A2} that $B' \in A_{\alpha}^p$ if and only if
\begin{equation} \label{kernelinA} 
\int_0^1 \int_0^{2\pi} \left( \frac{1 - |B(re^{i\theta})|}{1 - r} \right)^p (1-r)^{\alpha}\,d\theta\,dr < \infty,
\end{equation}
for $\alpha > -1$ and $p > 1 + \alpha$. In addition, we will use the relation (see \cite[p. 226]{T})
\[ \int_0^{2\pi} \frac{1}{|1 - re^{i\theta}|^x} \,d\theta \asymp  \frac{1}{(1 - r)^{h(x)}} \;\; (0 < r <1)\]
where $h(x) = \max(0, x - 1)$ and $x \in (0, 1) \cup (1, \infty)$.

Now since $1/2 < p \le 1$,
\begin{align*}
\biggl\{ \int_0^1 \int_0^{2\pi} &\left(\frac{1 - |B(re^{i\theta})|}{1 - r} \right)^p (1-r)^{\alpha}\,d\theta\,dr \biggr\}^{1/p}  \\
&\gtrsim \biggl\{\int_0^1 \int_0^{2\pi} \left(\sum_n |B_n(re^{i\theta})|^2 \frac{1 - |a_n|^2}{|1 - \bar{a}_nre^{i\theta}|^2}\right)^p(1-r)^{\alpha}\,d\theta\,dr \biggr\}^{1/p} \\
&\ge \sum_n \biggl\{\int_0^1 \int_0^{2\pi} |B_n(re^{i\theta})|^{2p} \frac{(1 - |a_n|^2)^p}{|1 - \bar{a}_nre^{i\theta}|^{2p}}(1-r)^{\alpha} \,d\theta\,dr \biggr\}^{1/p} \\
&\ge \sum_n \biggl\{\int_{r_n}^1 \int_0^{2\pi} |B_n(re^{i\theta})|^{2p} \frac{(1 - |a_n|)^p}{|1 - \bar{a}_nre^{i\theta}|^{2p}}(1-r)^{\alpha} \,d\theta\,dr \biggr\}^{1/p} \\
&\gtrsim \sum_n (1 - |a_n|) \biggl\{\int_{r_n}^1(1-r)^{\alpha} \int_0^{2\pi} \frac{1}{|1 - \bar{a}_nre^{i\theta}|^{2p}} \,d\theta\,dr \biggr\}^{1/p} \\
&\gtrsim \sum_n (1-|a_n|) \biggl\{ \int_{r_n}^1\frac{(1-r)^{\alpha}}{(1-|a_n|r)^{2p-1}}\,dr \biggr\}^{1/p} \\
&\gtrsim \sum_n (1-|a_n|) \biggl\{ \int_{r_n}^1\frac{(1-r)^{\alpha}}{((1-|a_n|) + (1-r))^{2p-1}}\,dr \biggr\}^{1/p}.
\end{align*}
But letting $d_n = 1- |a_n|$, we have
\begin{multline*}
 \int_{r_n}^1\frac{(1-r)^{\alpha}}{(d_n + (1-r))^{2p-1}}\,dr = d_n\int_0^{\frac{1-r_n}{d_n}} \frac{d_n^{\alpha}t^{\alpha}}{(d_n + d_n t)^{2p-1}}\,dt =\\ d_n^{1 + \alpha - (2p-1)}\int_0^{\frac{1-r_n}{d_n}}\frac{ t^{\alpha}}{(1+t)^{2p-1}}\,dt \ge d_n^{2 + \alpha - 2p}\int_0^{\frac{1-r_n}{d_n}} \frac{t^{\alpha}}{2}\,dt = \frac{(1-r_n)^{\alpha + 1}d_n^{1-2p}}{2\alpha+2}
\end{multline*}
since $(1 - r_n)/d_n \le 1$. As shown in \cite{AC2}, $1 - r_n \ge \frac{1}{4}d_n^{1+\beta}$. Thus,
\[\biggl\{ \int_0^1 \int_0^{2\pi} \left(\frac{1 - |B(re^{i\theta})|}{1 - r} \right)^p (1-r)^{\alpha}\,d\theta\,dr \biggr\}^{1/p} \gtrsim \sum_n (1-|a_n|)^{\frac{2 - p + \alpha\beta+\alpha +\beta}{p}}. \]
Since $B' \in A_{\alpha}^p$, the last series must converge, and so $2-p + \alpha\beta+\alpha +\beta \ge p\beta_0$. Letting $\beta \to \beta_0$, we get $2-p + \alpha\beta_0+\alpha +\beta_0 \ge p\beta_0$. This implies that $\beta_0 \le \frac{2 - p + \alpha}{p - \alpha -1}$. Thus, $\sum_n (1-|a_n|)^{\beta} < \infty$ for all $\beta > \frac{2 - p + \alpha}{p - \alpha -1}$.
\end{proof}

Theorem~\ref{sufficient} does not directly handle Bergman spaces, $A^p$. However, if $B' \in A^p$ for some $p \in (3/2, 2)$, then $B' \in H^{p-1}$ (see the discussion after the proof of Theorem~\ref{theoremHp}), and so $\sum_n (1-|a_n|)^{(2-p)/(p-1)} < \infty$ by a result of Ahern and Clark in \cite{AC1}.

\section{Uniformly Discrete Sequences}

In \cite{Ki}, Kim observed that if $B$ is a Blaschke product with zeros $\{a_n\}$ such that
\begin{equation} \label{sigma1}
\sum_n (1-|a_n|)^{2-p} < \infty,
\end{equation}
 then $B' \in A^p$, $1 < p < 2$. In the opposite direction, D.~Girela, J.~Pel\'{a}ez and D.~Vukoti\'{c} proved in \cite{GPV} the following result.

\begin{citedtheoremA}
Suppose $B$ is a Blaschke product with  zeros $\{a_n\}$ that are uniformly separated. If $B' \in A^p$, then condition~\eqref{sigma1} holds, $1 < p < 2$.
\end{citedtheoremA}

In \cite{P1}, it was proven that if $B$ is a Blaschke product with zeros $\{a_n\}$ such that
\begin{equation} \label{sigma2}
\sum_n (1-|a_n|)^{1-p} < \infty,
\end{equation}
then $B' \in H^p$, $1/2 < p < 1$. On the other hand, W.~Cohn proved in \cite{C} the following result.

\begin{citedtheoremB}
Suppose $B$ is a Blaschke product with  zeros $\{a_n\}$ that are uniformly separated. If $B' \in H^p$, then condition~\eqref{sigma2} holds, $0 < p < 1$.
\end{citedtheoremB}
Actually, Cohn states this result only for $1/2 < p < 1$, but his proof clearly holds for the larger range of $p$.

In \cite{K}, M.~Kutbi, proved for any Blaschke product $B$ with zeros $\{a_n\}$ that if $\sum_n (1 - |a_n|)^{\alpha} < \infty$, then
\begin{equation} \label{almostHp1}
\int_0^{2\pi} |B'(re^{it})|^p\,dt = o\left(\frac{1}{(1-r)^{p+\alpha -1}}\right),
\end{equation}
$0 < \alpha < 1/2$ and $p > 1 - \alpha$. This implication was extended in \cite{P2} to the case of $1/2 < \alpha \le 1$ and $p \ge \alpha$. Also in \cite{P2}, the following was proved.

\begin{citedtheoremC}
Let $B$ be a Blaschke product with  zeros $\{a_n\}$ that are uniformly separated, and suppose $0 < \alpha < 1$. Suppose there exists a positive number p such that
\begin{equation} \label{almostHp2}
\int_0^{2\pi} |B'(re^{it})|^p\,dt = O\left(\frac{1}{(1-r)^{p+\alpha -1}}\right),
\end{equation}
where $p > 1 - \alpha$ if $\,0 < \alpha \le 1/2$ and $p \ge \alpha$ if $1/2 < \alpha < 1$. Then, $\sum_n (1-|a_n|)^{\alpha'} < \infty$ for all $\alpha' > \alpha$.
\end{citedtheoremC}

It is the purpose of this section to extend Theorem~A, Theorem~B, and Theorem~C to cover Blaschke products whose zeros are uniformly discrete. This will be done using the following recent result by A.~Aleman and D.~Vukoti\'{c}, which appears in a more generalized form as Theorem 2(ii) in \cite{AV}.
\begin{citedtheoremD}
Let $B$ be a Blaschke product with  zeros $\{a_n\}$ that are uniformly discrete. Suppose $1/2 < p \le 1$ and $-1 < \alpha < 2p - 2$. If $B' \in A^p_{\alpha}$, then
\begin{equation} \label{sigma3}
\sum_n (1-|a_n|)^{2-p+\alpha} < \infty.
\end{equation}
\end{citedtheoremD}
We note that in Theorem 3 of \cite{P2}, the condition $p > (\alpha + 2)/2$ for $\alpha < 0$ is mistakenly omitted when stating that  \eqref{sigma3} implies $B' \in A^p_{\alpha}$.

\begin{theorem} \label{theoremAp}
Suppose $B$ is a Blaschke product with  zeros $\{a_n\}$ that are uniformly discrete. If $B' \in A^p$, then condition~\eqref{sigma1} holds, $1 < p < 2$.
\end{theorem}
\begin{proof}
In Theorem~5.1 of \cite{A1}, Ahern proved that $B' \in A^p$ if and only if $B' \in A_{1-p}^1$, $1 < p < 2$. We now are able to apply Theorem D since $-1<1-p<0$. Condition~\eqref{sigma1} holds since $2 - 1 + (1-p) = 2-p$.
\end{proof}

\begin{theorem} \label{theoremHp}
Suppose $B$ is a Blaschke product with  zeros $\{a_n\}$ that are uniformly discrete. If $B' \in H^p$, then condition~\eqref{sigma2} holds, $0 < p < 1$.
\end{theorem}
\begin{proof}
In Theorem~6.1 of \cite{A1}, Ahern proved for $0 < p < 1$ that
\[ \int_0^1 |B'(re^{it})|(1-r)^{-p}\,dr \le \frac{4}{p(1-p)}|B'(e^{it})|^p. \]
Thus, $B' \in H^p$ implies that $B' \in A_{-p}^1$. As in the last proof, it follows that \eqref{sigma2} holds. 
\end{proof}
We should note that Theorem~5.1 and Theorem~6.1 from \cite{A1} combine together to prove that $B' \in H^p$ implies $B' \in A^{p+1}$ for all $0 < p < 1$. On the other hand, Theorem~5.1 and Theorem~6.2 from \cite{A1}  prove that  $B' \in A^{p+1}$ implies $B' \in H^p$ for all $1/2 < p  < 1$. This last implication does not extend to $0 < p < 1$, as shown for $p = 1/2$ by Ahern and Clark in \cite{AC2} and  for $0 < p \le 1/2$ by Pel\'{a}ez in \cite{Pe}. Thus, Theorem~\ref{theoremAp} implies Theorem~\ref{theoremHp}, but Theorem~\ref{theoremHp} does not imply Theorem~\ref{theoremAp} for the full range of values of $p$.

In \cite{AC1}, Ahern and Clark proved that if $B$ is any Blaschke product with $B'  \in H^p$, then every subproduct of $B$ has its derivative in $H^p$, $0 < p < 1$. We will also need the corresponding result for $A^p_{\alpha}$ spaces.

\begin{lemma}
For any $\alpha > -1$ and $p > 1 + \alpha$, if $B$ is a Blaschke product such that  $B'  \in A^p_{\alpha}$, then every subproduct of $B$ also has its derivative in $A^p_{\alpha}$.
\end{lemma}
\begin{proof}
Suppose $B' \in A^p_{\alpha}$ and $\tilde{B}$ is any subproduct of $B$. Then condition~\eqref{kernelinA} holds for $B$, and so also for $\tilde{B}$. Thus, $\tilde{B}' \in A^p_{\alpha}$.\end{proof}

It follows easily from the lemma and from Ahern and Clark's result that Theorem~\ref{theoremAp} and Theorem~\ref{theoremHp} can be extended to Blaschke products whose zeros are thought of as unions of finitely many uniformly discrete sequences. Whether the content of each theorem can be extended to any even larger class of sequences of zeros is an open question. We note that some condition on a sequence $\{a_n\}$ beyond it being a Blaschke sequence is necessary; in \cite{AC1}, Ahern and Clark show that for each $p \in (1/2,1)$ there exists a Blaschke product $B$ such that $B' \in H^p$ but $\sum_n (1 - |a_n|)^{\beta} = \infty$ for all $\beta < (1-p)/p$.

\begin{theorem} \label{theoremAlmostHp}
Let $B$ be a Blaschke product with  zeros $\{a_n\}$ that are uniformly discrete, and suppose $0 < \alpha < 1$. Suppose there exists a positive number $p \le 1$ such that
\begin{equation}
\int_0^{2\pi} |B'(re^{it})|^p\,dt = O\left(\frac{1}{(1-r)^{p+\alpha -1}}\right),
\end{equation}
where $p > 1 - \alpha$ if $\,0 < \alpha \le 1/2$ and $p > \alpha$ if $1/2 < \alpha < 1$. Then, $\sum_n (1-|a_n|)^{\alpha'} < \infty$ for all $\alpha' > \alpha$.
\end{theorem}
\begin{proof}
For all $\alpha' > \alpha$,
\[ \int_0^{2\pi} |B'(re^{it})|^p(1 - r)^{p + \alpha' -2}\,dt = O\left((1-r)^{\alpha' - \alpha -1}\right). \]
Then,
\[ \iint_U  |B'(re^{it})|^p(1 - r)^{p + \alpha' -2}\,dA < \infty \]
since $\alpha' - \alpha > 0$. Notice that for any $0 < \alpha < 1$, we have $p > \alpha$. So, without loss of generality, we can assume that $\alpha' < p$. Then in all cases, $1/2 < p \le 1$ and $-1 < p + \alpha' - 2 < 2p - 2$. Thus, Theorem~D applies, and we conclude that $\sum_n (1-|a_n|)^{\alpha'} < \infty$.
\end{proof}

\end{document}